\documentclass[sn-mathphys-num]{sn-jnl}

\usepackage{graphicx}%
\usepackage{multirow}%
\usepackage{amsmath,amssymb,amsfonts}%
\usepackage{amsthm}%
\usepackage{mathrsfs}%
\usepackage[title]{appendix}%
\usepackage{xcolor}%
\usepackage{textcomp}%
\usepackage{manyfoot}%
\usepackage{booktabs}%
\usepackage{algorithm}%
\usepackage{algorithmicx}%
\usepackage{algpseudocode}%
\usepackage{listings}%
\usepackage{subfigure}
\usepackage{epstopdf}

\theoremstyle{thmstyleone}%
\newtheorem{theorem}{Theorem}[section]
 
\theoremstyle{thmstyletwo}%
\newtheorem{remark}{Remark}%

\theoremstyle{thmstylethree}%

\theoremstyle{thmstylefour}%
\newtheorem{lemma}{Lemma}[section]
\newtheorem{assumption}{Assumption}[section]

\newcommand{\T}{\text{T}}
\newcommand{\BB}{Barzilai-Borwein }

\begin{document}

\title[Article Title]{A Trust Region Method with  Regularized \BB Step-Size for Large-Scale Unconstrained Optimization}


\author[1]{\fnm{Xin} \sur{Xu}}\email{xustonexin@gmail.com}

\author*[1]{\fnm{Congpei} \sur{An}}\email{andbachcp@gmail.com}

\affil[1]{\orgdiv{School of Mathematics}, \orgname{Southwestern University of Finance and Economics}, \orgaddress{\city{Chengdu}, \country{China}}}



\abstract{We develop a Trust Region method with Regularized \BB step-size obtained in a previous paper for solving large-scale unconstrained optimization problems. Simultaneously, the non-monotone technique is combined to formulate an efficient trust region method. The proposed method adaptively generates a suitable step-size within the trust region. The minimizer of the resulted model can be easily determined, and at the same time, the convergence of the algorithm is also maintained. Numerical results are presented to support the theoretical results.}

\keywords{trust region, 
	Regularized \BB method, convergence}



\maketitle

\section{Introduction}\label{sec1}

In this paper, we consider a trust region method for solving the  unconstrained optimization problem
\begin{equation}\label{generalqua}
\min_{x\in\mathbb{R}^{n}} f(x),
\end{equation}
where $f:\mathbb{R}^n\longrightarrow \mathbb{R}$ is a real-valued twice-continuously differentiable function. A minimizer is denoted by $x_{*}$. Trust region method is iterative method. It has been shown that they are effective and robust for solving problem \eqref{generalqua}, especially for some non-convex, ill-conditioned problems \cite{Conn2000TrustRegionMethods,Yuan2015Recentadvancestrust}. A basic trust region algorithm works as follows. In the $k$-th iteration, the following trust region subproblem 
\begin{equation}\label{Powell}
\begin{aligned}	
&\min_{s\in\mathbb{R}^{n}}   m(x_{k}+s)=f(x_{k})+g_{k}^{\T}s+\frac{1}{2}s^{\T}H_{k}s \\
&\text{ s.t. } \ \|s\|\le \Delta_{k}
\end{aligned}
\end{equation}
is solved, where $g_{k}=\nabla f(x_{k})$, $H_{k}\in\mathbb{R}^{n\times n}$ is the Hessian or the Hessian approximation of $f$ at $x_{k}$, $\|\cdot\|$ refers to the Euclidean norm, and $\Delta_{k}>0$ is the trust region radius. A trial point 
\begin{equation}
\tilde{x}_{k+1}=x_{k}+s_{k}
\end{equation}
is then generated using an exact or an approximation solution $s_{k}$ of the trust region subproblem \eqref{Powell}. The ratio between the actual reduction and the predicted reduction 
\begin{equation}\label{rho}
\rho_{k}=\frac{\text{Ared}_{k}}{\text{Pred}_{k}}=\frac{f(x_{k})-f(\tilde{x}_{k+1})}{m(x_{k})-m(\tilde{x}_{k+1})}
\end{equation}
plays a critical role on deciding whether the trial step $s_{k}$ would be accepted or not and how the trust region radius would be adjusted for the next iteration. Specifically, this trial point is accepted as new iterate $x_{k+1}$ if a sufficient reduction of the true objective function is achieved, i.e., if $\rho_{k}\ge\eta_{1}$, where $\eta_{1}<1$ is a predefined positive threshold. If not, the trial point is rejected, and $x_{k+1}=x_{k}$. At the same time, the trust region radius expands or contracts, respectively. The usual empirical rules for this update can be summarized as follows (see, e.g., Conn et al. \cite{Conn2000TrustRegionMethods}):
\begin{equation}\label{radius}
\Delta_{k+1}=\begin{cases}
\alpha_{1}\Delta_{k}\quad &\text{if} \   \rho_{k}<\eta_{1},\\
\Delta_{k}\quad &\text{if} \   \eta_{1}\le\rho_{k}<\eta_{2},\\
\alpha_{2}\Delta_{k}\quad &\text{if} \   \rho_{k}\ge\eta_{2},
\end{cases}
\end{equation} 
where $\alpha_{1}$, $\alpha_{2}$, $\eta_{1}$, and $\eta_{2}$ are predefined constants such that
$$0<\eta_{1}\le\eta_{2}<1\quad \text{and} \quad 0<\alpha_{1}<1<\alpha_{2}.$$
The trust region method can be traced back to \text{Levenberg} \cite{Levenberg1944methodsolutioncertain} and \text{Marquardt} \cite{Marquardt1963AlgorithmLeastSquares} for solving nonlinear least square problems. Pioneer researches on trust region methods were given by \text{Powell} \cite{Powell1975Convergencepropertiesclass} and \text{Fletcher} \cite{Fletcher1970Efficientgloballyconvergent} et al. For more details on the trust region methods, we refer readers to \cite{Conn2000TrustRegionMethods,Yuan2015Recentadvancestrust} and references therein. 

In trust region methods, solving the subproblem \eqref{Powell} is a key issue \cite{Steihaug1983ConjugateGradientMethod,Liu1989limitedmemoryBFGS,Hager2001MinimizingQuadraticOver,Wang2006subspaceimplementationquasi,Zhou2016newsimplemodel}. These works attempt to solve the trust region subproblem by constructing approximations to the Hessian. The \BB (BB) method \cite{Barzilai1988TwoPointStep} obtains an effective gradient step-size with quasi-Newton properties through the secant equation and has the following form 
\begin{equation*}\label{gradient method}
x_{k+1}=x_{k}+\frac{1}{\alpha_{k}^{BB}}(-g_{k}),
\end{equation*}
where 
\begin{equation}\label{BBlong}
\alpha_{k}^{BB}=\frac{s_{k-1}^{\T}y_{k-1}}{s_{k-1}^{\T}s_{k-1}}\triangleq\alpha_{k}^{BB1},	
\end{equation}
or
\begin{equation}\label{BBshort}
\alpha_{k}^{BB}=\frac{y_{k-1}^{\T}y_{k-1}}{s_{k-1}^{\T}y_{k-1}}\triangleq\alpha_{k}^{BB2},
\end{equation}
where $s_{k-1}=x_{k}-x_{k-1}$ and $y_{k-1}=g_{k}-g_{k-1}$. The two scalars $\alpha_{k}^{BB1}$ and $\alpha_{k}^{BB2}$ in the BB method are the solutions to the following two least squares problems
\begin{flalign*}
\min_{\alpha>0}\|\alpha s_{k-1}- y_{k-1}\|^{2} \quad \mbox{and} \quad  \min_{\alpha>0}\|s_{k-1}-\alpha^{-1}y_{k-1}\|^{2}, 
\end{flalign*}
representing the secant equation at $x_{k}$. By the \text{Cauchy-Schwarz} inequality, we know that 
\begin{equation}\label{CSInBB}
\alpha_{k}^{BB1}\le\alpha_{k}^{BB2},
\end{equation}
if $s_{k-1}^{\T}y_{k-1}>0$. As is known to all, the BB method requires few storage locations and affordable computation operations. Due to its very satisfactory numerical performance and fast convergence ability \cite{Barzilai1988TwoPointStep,Yuan2018StepSizesGradient}, there are several intensive researches about BB method. For example, Raydan \cite{Raydan1997BarzilaiBorweinGradienta} proposed a globalized BB gradient method with non-monotone line search for solving large scale optimization problems. Zhou et al. \cite{Zhou2006GradientMethodsAdaptive} presented an adaptive BB method. Dai et al. \cite{Dai2006cyclicBarzilaiBorwein,Dai2019familyspectralgradient} developed a cyclic BB method and a family spectral gradient method for unconstrained optimization. Burdakov et al \cite{Burdakov2019StabilizedBarzilaiBorweina} introduced a stabilized BB method in order to solve the phenomenon that BB frequently produces long steps. Huang et al. \cite{Huang2021EquippingBarzilaiBorwein} designed a BB method with the two dimensional quadratic termination property. Ferrandi et al. \cite{Ferrandi2023harmonicframeworkstepsize} proposed a harmonic framework for stepsize selection in gradient method.      

An interesting point, solving the trust region subproblem \eqref{Powell} is equivalent to solving a regularized subproblem
\begin{equation}\label{Regular}
\min_{s\in\mathbb{R}^{n}}   m(x_{k}+s)=f(x_{k})+g_{k}^{\T}s+\frac{1}{2}s^{\T}H_{k}s+\frac{1}{2}\lambda_{k}\|s\|^{2}, 
\end{equation}
for a proper regularization parameter $\lambda_{k}\ge0$. That is,
\begin{equation}\label{LM}
s_{k}=-\big(H_{k}+\lambda({\Delta_{k}})I\big)^{-1}g_{k},
\end{equation}
where $\lambda({\Delta_{k}})$ is the regularization parameter, $I$ is the identity matrix. Suppose $H_{k}$ is positive definite. We have the following result. 
\begin{enumerate}
	\item[(1)]If $\|H_{k}^{-1}g_{k}\|\le\Delta_{k}$, then $\lambda({\Delta_{k}})=0$, $s_{k}=H_{k}^{-1}g_{k}$;
	\item[(2)]otherwise, increase $\lambda({\Delta_{k}})$ so that $\|\big(H_{k}+\lambda(\Delta_{k})I\big)^{-1}g_{k}\|\le\Delta_{k}$ holds.
\end{enumerate}
Thus, trust region algorithm can be viewed as a regularization by using the quadratic penalty function \cite{Yuan2015Recentadvancestrust}. \cite{Cartis2009Adaptivecubicregularisation} proposed an adaptive regularization algorithm using cubic (ARC) for unconstrained optimization. At each iteration of the ARC algorithm, the following cubic model
$$\min_{s\in\mathbb{R}^{n}}   m(x_{k}+s)=f(x_{k})+g_{k}^{\T}s+\frac{1}{2}s^{\T}H_{k}s+\frac{1}{3}\lambda_{k}\|s\|^{3}, $$
is used to obtain the trial step $s_{k}$. In such a framework, the $\lambda_{k}$ in the ARC algorithm might be regarded as inversely proportional to the trust-region radius $\Delta_{k}$ \cite[Theorem 3.1]{Cartis2009Adaptivecubicregularisation}. This inverse relationship between the regularization parameters and the trust region radius can also be observed from \eqref{LM}. Thus $\lambda_{k}$ is increased if insufficient decrease is obtained in some measure of relative objective change, but decreased or unchanged otherwise.

In this paper, we propose a trust region method based on the BB method and regularization idea, which could be regard as a regularized \BB (RBB) method. In our method, the regularized BB stepsize is embedded in the framework of a simple trust region model, that is, the Hessian approximation $H_{k}$ is taken as a real positive definite scalar matrix $\alpha_{k}^{RBB}I$ for some $\alpha_{k}^{RBB}>0$, which inherit certain quasi-Newton properties. And the regularization parameter in RBB method depends on the size of trust region radius, so that the size of the RBB stepsize can truly reflect the local Hessian information, making the RBB stepsize integrated with the trust region framework. 

Then, the trust region subproblem \eqref{Powell} can be written as 
\begin{equation}\label{SimpleTRU}
\begin{aligned}	
&\min_{s\in\mathbb{R}^{n}}   m(x_{k}+s)=f(x_{k})+g_{k}^{\T}s+\frac{1}{2}\alpha_{k}^{RBB}s^{\T}s \\
&\text{ s.t. } \ \|s\|\le \Delta_{k}.
\end{aligned}
\end{equation} 
Suppose $\|g_{k}\|\neq 0$. 	If $\|\frac{1}{\alpha_{k}^{RBB}}g_{k}\|\le\Delta_{k}$, then $s_{k}=-\frac{1}{\alpha_{k}^{RBB}}g_{k}$. Otherwise, the optimal solution to \eqref{SimpleTRU} will be on the boundary of the trust region \cite{Sun2006OptimizationTheoryMethods}, i.e., $s_{k}$ is the solution to the subproblem \eqref{SimpleTRU} with $\|s\|=\Delta_{k}$. Solving the problem, we have the solution $s_{k}=-\frac{\Delta_{k}}{\|g_{k}\|}g_{k}$. Therefore, the solution to subproblem \eqref{SimpleTRU} is
	\begin{equation}\label{tk}
	s_{k}=-t_{k}g_{k},
	\end{equation}
	where $t_{k}=\min\{\frac{1}{\alpha_{k}^{RBB}},\ \frac{\Delta_{k}}{\|g_{k}\|}\}$.

For some problems, non-monotonic strategies can overcome the so-called \text{Maratos} Effect \cite{Maratos1978Exactpenaltyfunction}, which could
lead to the rejection of \text{superlinear} convergent steps since these steps could cause an increase in both the objective function value and the constraint violation, and improve the performance of the algorithms. The traditional trust region methods are monotonic. As a prominent example, in order to improve the efficiency of the trust region methods, the first exploitation of non-monotone strategies in a trust-region framework was proposed in \cite{Deng1993Nonmonotonictrustregion} by modifying the definition of ratio $\rho_{k}$ in \eqref{rho} to assess an agreement between the objective function reduction and the model predicted reduction over the trust region, that is 
\begin{equation}\label{nonrho}
\rho_{k}=\frac{\text{Ared}_{k}}{\text{Pred}_{k}}=\frac{f(x_{l(k)})-f(\tilde{x}_{k+1})}{m(x_{k})-m(\tilde{x}_{k+1})},
\end{equation} 
where $f(x_{l(k)})\ge f(x_{k})$ is a reference value. There are various non-monotonic techniques, i.e., generating different types of reference values $f(x_{l(k)})$, for example, see \cite{Zhang2004NonmonotoneLineSearch,Ahookhosh2011efficientnonmonotonetrust,Esmaeili2014IMPROVEDADAPTIVETRUST}. In the trust region framework, we adopt the non-monotone strategy in \cite{Grippo1991classnonmonotonestabilization}, where 
\begin{equation}\label{Ck}
f(x_{l(k)})=\max_{0\le j\le n(k)}\{f(x_{k-j})\}
\end{equation}
with $n(k)=\min\{M, k\}$, $M\ge0$ is an integer.

The remaining part of this paper is organized as follows. In Section 2, we propose a  trust region algorithm with Regularized \BB step-size for solving large-scale unconstrained optimization problems. In Section 3, we analyze the convergence properties of the trust region algorithm. In Section 4, we present numerical results of the proposed algorithm for an unconstrained optimization test collection and challenging spherical $t$-designs. Finally, in Section 5, we present concluding remarks.

\section{Trust region method with RBB step-size }\label{TRBB}
In this section, we propose several strategies on how to determine the scalar $\alpha_{k}^{RBB}$, which would be a critical issue for the success of our algorithm. It is well know that the classic quasi-Newton equation
\begin{equation}\label{quasi-Newton}
	H_{k}s_{k-1}=y_{k-1},
\end{equation} 
where $s_{k-1}=x_{k}-x_{k-1}$, $y_{k-1}=g_{k}-g_{k-1}$. It is usually difficult to satisfy the quasi-Newton equation \eqref{quasi-Newton} with a nonsingular scalar matrix \cite{Zhou2016newsimplemodel}. Hence, we need some alternative conditions that can maintain the accumulated curvature information along the negative gradient as correct as possible. Some authors have proposed modified quasi-Newton equations, for example, see  \cite{Zhang1999NewQuasiNewton,Hassan2020newquasiNewton,Hassan2023modifiedsecantequation}. 

 An and Xu \cite{An2024RegularizedBarzilaiBorwein} proposed a regularized BB framework by adding regularization term to the BB least squares model. Assuming $H_{k}$ is symmetric positive definite, the regularized BB framework in \cite{An2024RegularizedBarzilaiBorwein} as follows:
\begin{equation*}\label{equ: least square with regular}
\min_{\alpha>0} \Big\{\Vert \alpha s_{k-1} - y_{k-1}\Vert^{2} + \tau_{k}\Vert  \alpha \Phi(H_{k})s_{k-1} - \Phi(H_{k}) y_{k-1}\Vert^{2} \Big\},
\end{equation*}
where $\tau_{k}\ge0$ is the  regularization parameter, $\Phi$ is an operation acting on $H_{k}$. 
Then the solution to above regularized least squares problem as follows
\begin{equation}\label{equ: regBBstepsize}
\alpha_{k}^{new}=\frac{s_{k-1}^{\mathrm{T}}y_{k-1}+\tau_{k} s_{k-1}^{\T}\Phi(H_{k})^{\T}\Phi(H_{k})y_{k-1}}{s_{k-1}^{\mathrm{T}}s_{k-1}+\tau_{k} s_{k-1}^{\T}\Phi(H_{k})^{\T}\Phi(H_{k})s_{k-1}}.
\end{equation}
Selecting the appropriate operator $\Phi$ to modify the original BB least squares model is the core of the algorithm. Below, we present a scheme for selecting the operator $\Phi$, so that the modified model can generate scalar matrices that approximate Hessian. 

We now consider 
\begin{equation}\label{PHi}
	\Phi(H_{k})=\sqrt{H_{k}},
\end{equation}
then 
\begin{equation}\label{new}
	\alpha_{k}^{new}=\frac{s_{k-1}^{\T}(I+\tau_{k}H_{k})y_{k-1}}{s_{k-1}^{\T}(I+\tau_{k}H_{k})s_{k-1}}=\frac{s_{k-1}^{\mathrm{T}}y_{k-1}+\tau_{k} y_{k-1}^{\T}y_{k-1}}{s_{k-1}^{\mathrm{T}}s_{k-1}+\tau_{k} s_{k-1}^{\T}y_{k-1}}.
\end{equation}
Recalling the two BB scalars, from the basic inequality, we have 
\begin{equation}\label{relation}
\alpha_{k}^{BB1}\le\alpha_{k}^{new}\le\alpha_{k}^{BB2}.	
\end{equation} 
Then, $\alpha_{k}^{new}$ remains within the spectrum of $H_{k}$ and possess quasi-Newton properties. An important property of $\alpha_{k}^{new}$ in \eqref{new} is that it is monotonically increasing with respect to $\tau_{k}$, which guides us in choosing the appropriate regularization parameter. 
The success of the \text{ABBmin} scheme in \cite{Frassoldati2008Newadaptivestepsize} inspires us to use the following adaptively alternate scheme
\begin{equation}\label{alternate step}
\alpha_{k}^{RBB}=\begin{cases}
\max\{\alpha_{j}^{new}|j\in\{j_{0},\,\ldots,\,k\}\} & \text{if}\quad \alpha_{k}^{BB1}/\alpha_{k}^{BB2}<\nu_{k},\\
\alpha_{k}^{BB1}& \text{otherwise},
\end{cases}
\end{equation} 
where $j_{0}=\max\{1,\,k-\varrho\}$, $\varrho\ge0$ is an integer,
\begin{equation*}
\nu_{k}=1-\frac{\alpha_{k}^{BB1}}{\alpha_{k}^{new}}\in(0, 1).	
\end{equation*}
From \eqref{relation}, we know $0<\alpha_{k}^{BB1}/\alpha_{k}^{BB2}\le \alpha_{k}^{BB1}/\alpha_{k}^{new}<1$. If $\alpha_{k}^{BB1}/\alpha_{k}^{BB2}<\nu_{k}$, then we have $\alpha_{k}^{BB1}/\alpha_{k}^{BB2}<0.5$. Based on the idea of ABB \cite{Zhou2006GradientMethodsAdaptive}, in this case, it makes sense to choose a short stepsize.  
		 
\subsection{Selection of parameter $\tau_{k}$ and update of trust-region radius $\Delta_{k}$}

Observing the solution $s_{k}$ in \eqref{tk}, we can see that the regularization parameter $\tau_{k}$ of $\alpha_{k}^{RBB}$ and the trust region radius $\Delta_{k}$ determine the performance of the new method. As shown in \eqref{LM}, the size of the trust region radius provides a reference for the selection of regularization parameters $\tau_{k}$ in \eqref{equ: regBBstepsize}. Based on this, we only need to refine the update of the trust region radius $\Delta_{k}$.

According to $\rho_{k}$, the iterations defined in \eqref{radius} are called failed, successful and very successful, respectively. Correspondingly, the trust region radius $\Delta_{k}$ shrinks, remains unchanged, and expands. Intuitively, the smaller the trust region radius, the higher the accuracy of the subproblem model \eqref{SimpleTRU}. In the case of very successful iteration, the reason for increasing $\Delta_{k}$ is that we are confident that the model is accurate within a large range at the current iterate. Therefore, an algorithm should be allowed to take larger trust region radius if needed. Conversely, in the case of failed iteration, shorter trust region radius should be allowed. However, in very successful iterations, if $\rho_{k}$ is significantly greater than $1$, then the degradation of the objective function is likely to depend on an inaccurate local model. In this case, the decrease in the objective function appears rather fortunate. Therefore, \cite{Walmag2005NoteTrustRegion} suggests a \textit{too successful} iteration, i.e., $\rho_{k}>\eta_{3}>1$. At this time, the expansion coefficient $\alpha_{3}$ of the trust region radius $\Delta_{k}$ should be less than that of the $\rho_{k}$ close to $1$, i.e., $1<\alpha_{3}<\alpha_{2}$. 

Inspired by the ``too successful" iteration above, we now consider the other extreme case, a \textit{too failed} iteration, i.e., $0<\rho_{k}<\eta_{4}<\eta_{1}$. In this case, it is reasonable to continue shrink the size of the trust region radius compared to the failed iteration, i.e., $0<\alpha_{4}<\alpha_{1}$. These discussions suggest that we adopt a refined update rule of the trust region radius, and the replacement of \eqref{radius} by
\begin{equation}\label{refined radius}
\Delta_{k+1}=\begin{cases}
\alpha_{4}\Delta_{k}\quad &\text{if} \   \rho_{k}<\eta_{4},\\
\alpha_{1}\Delta_{k}\quad &\text{if} \   \eta_{4}\le\rho_{k}<\eta_{1},\\
\Delta_{k}\quad &\text{if} \   \eta_{1}\le\rho_{k}<\eta_{2},\\
\alpha_{2}\Delta_{k}\quad &\text{if} \   \eta_{2}\le\rho_{k}<\eta_{3},\\
\alpha_{3}\Delta_{k}\quad &\text{if} \   \rho_{k}\ge\eta_{3},
\end{cases}
\end{equation}  
where
\begin{equation*}\label{eta}
0<\eta_{4}<\eta_{1}<\eta_{2}<1<\eta_{3},
\end{equation*}
and
\begin{equation*}\label{scale}
0<\alpha_{4}<\alpha_{1}<1<\alpha_{3}<\alpha_{2}.
\end{equation*}

This is, however, only part of the story. Another main purpose of discussing the update rule of the refined trust region radius is to design effective regularization parameters in $\alpha_{k}^{RBB}$, 
from \eqref{LM}, \eqref{equ: regBBstepsize} and \eqref{tk}, we know that there is an inverse proportional relationship between the regularization parameters $\tau_{k}$ and the trust region radius $\Delta_{k}$. Then, we obtain the following update rule of the regularization parameter
\begin{equation}\label{parameters}
\tau_{k}=\Delta_{k}^{-1}.
\end{equation}
Reasonably, we can also choose other functions with non-negative monotonic properties to generate the regularization parameters. For example, a regularization parameter can also have the following exponential form
\begin{equation}\label{exp}
\tau_{k}=e^{-\Delta_{k}}.
\end{equation} 
On these two simple elementary functions, for any $\Delta_{k}\ge0$, $e^{-\Delta_{k}}<\Delta_{k}^{-1}$ is established, and 	$\Delta_{k}^{-1}\in(0,\infty)$ and $e^{-\Delta_{k}}\in\left(0,1\right]$ are hold. Due to $e^{-\Delta_{k}}\in\left(0,1\right]$, it is equivalent to fine-tuning the original BB1 step-size. Intuitively, if the current trust region radius is large, it indicates that the trust region subproblem  model is consistent with the actual target in a large range, and choosing a long step size helps the algorithm converge quickly. This means that we should reduce the penalty for BB1 and choose a small regularization parameter $\tau_{k}$. Otherwise, a larger regularization parameter should be selected.  
\begin{remark}
	Schemes \eqref{parameters} and \eqref{exp} represent two types of regularization parameters, \eqref{parameters} can be generalized to $\tau_{k}=\Delta_{k}^{-m}$ where $m\ge 1$ is an integer, and \eqref{exp} can be generalized to $\tau_{k}=a^{-\Delta_{k}}$ where $a>1$ is a constant.
\end{remark}	

Now we state the trust region algorithm with Regularized \BB step-size for solving large-scale unconstrained optimization as follows.
\begin{algorithm}[h!]
	\caption{Trust Region method with  Regularized \BB step-size.}\label{alg:TRBB}
	\begin{algorithmic}[1]
		\Require stopping criterion: $\varepsilon>0$, ${\text{MaxIt}>0}$;\\
		Initialization: $x_{1}\in\mathbb{R}^{n}$, $t_{\rm{max}}\ge t_{\rm{min}}>0$, initial step-size $t_{1}\in (t_{min},\,t_{max})$, $0<\eta_{4}\le\eta_{1}\le\eta_{2}<1<\eta_{3}$, $0<\alpha_{4}<\alpha_{1}<1<\alpha_{3}<\alpha_{2}$, $\vartheta$, $\varrho$, $M\in\mathbb{N}^{+}$, regularization parameter and trust region radius: $\tau_{1}$, $\Delta_{1}$. $k=1$.
		\While{$k<\rm{MaxIt}$ {\text{and}} $\|g_{k}\|_{2}>\varepsilon\|g_{1}\|_{2}$}
		{Solve the subproblem \eqref{SimpleTRU} for $s_{k}$.\\
			Compute $\rho_{k}$ by \eqref{nonrho}.\\
			Set
			\begin{equation}\label{update}
			x_{k+1}=\begin{cases}
			x_{k}+s_{k}\quad &\text{if} \ \rho_{k}\ge\eta_{1},\\
			x_{k}\quad &\text{otherwise}.
			\end{cases}
			\end{equation}\\
			Set $\Delta_{k+1}$ by $\eqref{refined radius}$ and $\tau_{k+1}$ by  $\eqref{parameters}$ or $\eqref{exp}$, respectively.\\
			$k=k+1$.}
		\EndWhile
	\end{algorithmic}
\end{algorithm}

When negative curvature is detected in  some iterations, that is, $s_{k-1}^{\T}y_{k-1}\le0$, in which case $\alpha_{k}^{new}=\frac{\|y_{k-1}\|}{\|s_{k-1}\|}$ will be used. Meanwhile, we restrict $1/\alpha_{k}^{RBB}$ to the interval $[t_{min},\ t_{max}]$. 
\begin{remark}
	The number of failed iterations from $x_{k}$ to $x_{k+1}$ is called the internal circulation. which means the computation times of the subproblem between the current and the next iterates.	
\end{remark}

\section{Convergence analysis}
In this section, we discuss the global convergence of Algorithm \ref{alg:TRBB}. This analysis follows the convergence analysis framework of the trust region methods (see \cite{Conn2000TrustRegionMethods,Sun2006OptimizationTheoryMethods,1999NumericalOptimization,Zhou2016newsimplemodel}).

Throughout the paper, we consider the following assumption in order to analyze the convergence of the proposed Algorithm \ref{alg:TRBB}.
\begin{assumption}\label{assumption1}
	The level set $L(x_{1})=\{x|f(x)\le f(x_{1})\}$ is bounded and $f(x)$ is continuously differentiable in $L(x_{1})$ for any given $x_{1}\in\mathbb{R}^{n}$.	
\end{assumption}

\begin{lemma}\label{residual}
	Assume that $s_{k}$ is a solution of \eqref{SimpleTRU}, then $f(x_{k})-f(x_{k}+s_{k})-\text{Pred}_{k}=\mathcal{O}(\|s_{k}\|^{2})$.
\end{lemma}
\begin{proof}
	From the definition of $\text{Pred}_{k}$ and Taylor expansion, this is trivial.	
\end{proof}
\begin{lemma}\label{WellDe}
	Suppose $\|g_{k}\|\neq 0$. The solution $s_{k}$ of the model \eqref{SimpleTRU}
	satisfies
	\begin{equation}\label{WellDefine}
	\text{Pred}_{k}=m(x_{k})-m(x_{k}+s_{k})\ge \frac{1}{2}\|g_{k}\|\min\{\Delta_{k}, \frac{\|g_{k}\|}{\alpha_{k}^{RBB}}\}.
	\end{equation}
\end{lemma}
\begin{proof}
	If $\frac{\|g_{k}\|}{\alpha_{k}^{RBB}}\le\Delta_{k}$, then $s_{k}=-\frac{1}{\alpha_{k}^{RBB}}g_{k}$. Hence, we have
	\begin{equation}\label{first}
	\begin{split}
	\text{Pred}_{k}&=f(x_{k})-\big(f(x_{k})+g_{k}^{\T}s_{k}+\frac{1}{2}\alpha_{k}^{RBB}s_{k}^{\T}s_{k}\big)\\
	&=-g_{k}^{\T}(-\frac{1}{\alpha_{k}^{RBB}}g_{k})-\frac{1}{2}\alpha_{k}^{RBB}(\frac{1}{\alpha_{k}^{RBB}})^{2}g_{k}^{\T}g_{k}\\
	&=\frac{1}{2}\frac{1}{\alpha_{k}^{RBB}}\|g_{k}\|^{2}.
	\end{split}
	\end{equation} 	
	If $\frac{\|g_{k}\|}{\alpha_{k}^{RBB}}>\Delta_{k}$, then $s_{k}=-\frac{\Delta_{k}}{\|g_{k}\|}g_{k}$. We have 
	\begin{equation}\label{second}
	\begin{split}
	\text{Pred}_{k}&=f(x_{k})-(f(x_{k})+g_{k}^{\T}s_{k}+\frac{1}{2}\alpha_{k}^{RBB}s_{k}^{\T}s_{k})\\
	&=-g_{k}^{\T}(-\frac{\Delta_{k}}{\|g_{k}\|}g_{k})-\frac{1}{2}\alpha_{k}^{RBB}\frac{\Delta_{k}^{2}}{\|g_{k}\|^{2}}\|g_{k}\|^{2}\\
	&=\frac{\Delta_{k}}{\|g_{k}\|}\|g_{k}\|^{2}-\frac{1}{2}\alpha_{k}^{RBB}\Delta_{k}^{2} \\
	&>\frac{1}{2}\|g_{k}\|\Delta_{k}.
	\end{split}
	\end{equation}
	It follows from \eqref{first} and \eqref{second} that \eqref{WellDefine} holds.  
\end{proof}
\begin{remark}
	Lemma \ref{WellDe} shows that the $\rho_{k}$ in \eqref{nonrho} is well-defined, and the solution to the subproblem could always reduce the function value of the subproblem.
\end{remark}
\begin{lemma}\label{inner iteration}
	If $\|g_{k}\|\neq 0$, then Algorithm \ref{alg:TRBB} is well-posed, that is, Algorithm \ref{alg:TRBB} won't cycle infinitely in internal circulation.
\end{lemma}
\begin{proof}
	We prove the lemma by way of contradiction. At the current iterate $x_{k}$, we denote the inner cycling index at $x_{k}$ by $k(i)$, the solution of \eqref{SimpleTRU} by $s_{k(i)}$ and the corresponding predicted reduction by $\text{Pred}_{k(i)}$. Then we have 
	\begin{equation}\label{contra}
	\rho_{k(i)}<\eta_{1}\quad \text{for all} \quad i=1,2,\ldots
	\end{equation}
	and $\Delta_{k(i)}\rightarrow0$ as $i\rightarrow\infty$.
	
	By Lemmas \ref{residual} and \ref{WellDe}, for $i$ large enough, we have
	\begin{equation}
	\begin{split}
	\Big|\frac{f(x_{k})-f(x_{k}+s_{k(i)})}{m(x_{k})-m(x_{k}+s_{k(i)})}-1\Big|&=\Big|\frac{f(x_{k})-f(x_{k}+s_{k(i)})-\text{Pred}_{s_{k(i)}}}{\text{Pred}_{k(i)}}\Big|\\
	&=\frac{\mathcal{O}(\|s_{k(i)}\|^{2})}{\text{Pred}_{k(i)}}\\
	&\le\frac{\mathcal{O}(\|s_{k(i)}\|^{2})}{\frac{1}{2}\|g_{k}\|\min\{\Delta_{k(i)},\frac{\|g_{k}\|}{\alpha^{RBB}}\}}\\
	&=\frac{\mathcal{O}(\|s_{k(i)}\|^{2})}{\frac{1}{2}\|g_{k}\|\Delta_{k(i)}}\\
	&\le\frac{\mathcal{O}(\Delta_{k(i)})}{\frac{1}{2}\|g_{k}\|}\rightarrow 0,
	\end{split}
	\end{equation}	
	which implies 
	\begin{equation}\label{c}
	\lim_{i \to \infty}\frac{f(x_{k})-f(x_{k}+s_{k(i)})}{\text{Pred}_{k(i)}}=1.
	\end{equation}
	Combining \eqref{nonrho}, \eqref{c} and Lemma \ref{monotone}, we have
	\begin{equation*}
	\rho_{k(i)}=\frac{f_{l(k)}-f(x_{k}+s_{k(i)})}{\text{Pred}_{k(i)}}\ge\frac{f(x_{k})-f(x_{k}+s_{k(i)})}{\text{Pred}_{k(i)}}.
	\end{equation*}
	This inequality together with \eqref{c} imply that $\rho_{k(i)}\ge\eta_{1}$ for sufficiently large $i$, which contradicts \eqref{contra}. We complete the proof.
\end{proof}
\begin{lemma}\label{level set}
	Suppose $\{x_{k}\}$ is an infinite sequence generated by Algorithm \ref{alg:TRBB}. Then $\{x_{k}\}\subset L(x_{1})$.
\end{lemma}
\begin{proof}
	By the induction on $k$. If $k=1$, $x_{1}\in L(x_{1})$ obviously. Assume that $x_{k}\in L(x_{1})$ for $k>1$. From \eqref{update} in Algorithm \ref{alg:TRBB} and \eqref{nonrho}, we have
	\begin{equation}\label{fe}
	f(x_{k+1})\le f(x_{l(k)})-\eta_{1}\text{Pred}_{k}
	\end{equation}
	holds.
	By the definition of $f(x_{l(k)})$ in \eqref{Ck} and the induction on $k$, we have $l(k)\le k$ and $f_{l(k)}\le f(x_{1})$. Thus, by \eqref{fe}, we have $f(x_{k+1})<f(x_{1})$, which implies  $x_{k+1}\in L(x_{1})$. We complete the proof. 
\end{proof}
\begin{lemma}\label{monotone}
	Suppose $\{x_{k}\}$ is an infinite sequence generated by Algorithm \ref{alg:TRBB}. Then the sequence $\{f(x_{l(k)})\}$ is monotonic non-increasing and convergent.
\end{lemma}
\begin{proof}
	By Algorithm \ref{alg:TRBB} and Lemma \ref{WellDe}, we have 
	\begin{equation}\label{fr}
	f(x_{l(k)})\ge f(x_{k+1})+\eta_{1}\text{Pred}_{k}\ge f(x_{k+1})\quad\text{for all}\quad k\ge 1.
	\end{equation}
	We consider the following two cases.
	
	If $k<M$, then $n(k)=\min\{M, k\}=k$, and 
	$$f(x_{l(k)})=\max_{0\le j\le n(k)}\{f(x_{k-j})\}=f(x_{1}).$$
	
	If $k\ge M$, then $n(k)=\min\{M, k\}=M$, $n(k+1)=\min\{M, k+1\}=M$. By the definition of $f(x_{l(k)})$ in \eqref{Ck} and \eqref{fr}, we have
	\begin{equation*}
	\begin{split}
	f(x_{l(k+1)})=\max_{0\le j\le n(k+1)}\{f(x_{k+1-j})\}&=\max\Big\{\max_{0\le j\le n(k)-1}\{f(x_{k-j})\},  f(x_{k+1})\Big\}\\
	&\le\max\{f(x_{l(k)}), f(x_{k+1})\}\le f(x_{l(k)}).
	\end{split}
	\end{equation*}
	Therefore, the sequence $\{f_{l(k)}\}$ is monotonic non-increasing. It follows from Assumption \ref{assumption1} and Lemma \ref{level set} that $\{f(x_{l(k)})\}$ is bounded, which implies that $\{f(x_{l(k)})\}$ is convergent.
\end{proof}

\begin{theorem}
	Let $\{x_{k}\}$ be an infinite sequence generated by Algorithm \ref{alg:TRBB}. Then we have 
	\begin{equation}\label{liminf}
	\liminf_{k\rightarrow\infty} \|g_{k}\|=0.
	\end{equation}
\end{theorem}
\begin{proof}
	We prove \eqref{liminf} by way of contradiction. Assume that there exists a positive constant $\epsilon_{0}$ such that 
	\begin{equation}\label{eps}
	\|g_{k}\|\ge\epsilon_{0} \quad\text{for all}\quad k\ge 1.
	\end{equation}
	It follows from \eqref{update} in Algorithm \ref{alg:TRBB}, \eqref{nonrho} and Lemma \ref{monotone} that
	\begin{equation}\label{d}
	f(x_{k+1})\le f(x_{l(k)})-\eta_{1}\text{Pred}_{k}\le f(x_{l(k)})-\frac{1}{2}\eta_{1}\|g_{k}\|\min\{\Delta_{k}, \frac{\|g_{k}\|}{\alpha_{k}^{RBB}}\}.
	\end{equation} 
	Combining \eqref{eps} and \eqref{d}, we have  
	\begin{equation}\label{lk}
	\begin{split}
	f(x_{l(l(k))})-f(x_{l(k)+1})&\ge \eta_{1}\text{Pred}_{l(k)}\\
	&\ge\frac{1}{2}\eta_{1}\|g_{l(k)}\|\min\{\Delta_{l(k)}, \frac{\|g_{l(k)}\|}{\alpha_{l(k)}^{RBB}}\}\\
	&\ge\frac{1}{2}\eta_{1}\epsilon_{0}\min\{\Delta_{l(k)}, \frac{\epsilon_{0}}{\alpha_{max}}\}
	\end{split}
	\end{equation}
	Because $\{f(x_{l(k)})\}$ is convergent, the above inequality implies that there exists an infinite index set $\mathcal{S}$ such that 
	\begin{equation}\label{Deltalk}
	\lim\limits_{l(k)\rightarrow\infty,l(k)\in\mathcal{S}}\Delta_{l(k)}=0.
	\end{equation}
	
	Without loss of generality, we can assume that for all $l(k)\in\mathcal{S}$ there are more than one inner cycle performed at the $l(k)$-th iterate.	
	So, the solution $s_{l(k)}$ of the following subproblem
	\begin{equation*}
	\begin{aligned}	
	&\min_{s\in\mathbb{R}^{n}}   m(x_{l(k)}+s)=f(x_{l(k)})+g_{l(k)}^{\T}s+\frac{1}{2}\alpha_{l(k)}^{RBB}s^{\T}s \\
	&\text{ s.t. } \ \|s\|\le\frac{\Delta_{l(k)}}{\eta_{1}},\quad k\in\mathcal{S},
	\end{aligned}
	\end{equation*} 
	is not accepted at the $l(k)$-th iterate, which means 
	\begin{equation}\label{re}
	\rho_{s_{l(k)}}=\frac{f(x_{l(l(k))})-f(x_{l(k)}+s_{l(k)})}{m(x_{l(k)})-m(x_{l(k)}+s_{l(k)})}<\eta_{1}\quad \text{for all}\quad l(k)\in\mathcal{S}.
	\end{equation}
	On the other hand, from \eqref{eps}, \eqref{Deltalk}, Lemmas \ref{residual} and \ref{WellDe}, we have 
	\begin{equation}
	\begin{split}
	\Big|\frac{f(x_{l(k)})-f(x_{l(k)}+s_{l(k)})}{m(x_{l(k)})-m(x_{l(k)}+s_{l(k)})}-1\Big|&=\frac{\mathcal{O}(\|s_{l(k)}\|^2)}{m(x_{l(k)})-m(x_{l(k)}+s_{l(k)})}\\
	&\le\frac{\mathcal{O}(\|s_{l(k)}\|^2)}{\frac{1}{2}\|g_{l(k)}\|\min\{\frac{\Delta_{l(k)}}{\eta_{1}},\frac{\|g_{l(k)}\|}{\alpha_{max}}\}}\\
	&\le\frac{\mathcal{O}(\frac{\Delta_{l(k)}}{\eta_{1}})}{\frac{1}{2}\|g_{l(k)}\|}\rightarrow 0,\quad l(k)\in\mathcal{S},
	\end{split}
	\end{equation}	
	which implies 
	\begin{equation}\label{e}
	\lim_{l(k)\to \infty,l(k)\in\mathcal{S}}\frac{f(x_{l(k)})-f(x_{l(k)}+s_{l(k)})}{m(x_{l(k)})-m(x_{l(k)}+s_{l(k)})}=1.
	\end{equation}
	Combining \eqref{e} and Lemma \ref{monotone}, we have
	\begin{equation*}
	\rho_{s_{l(k)}}=\frac{f(x_{l(l(k))})-f(x_{l(k)}+s_{l(k)})}{m(x_{k})-m(x_{l(k)}+s_{l(k)})}\ge\frac{f(x_{l(k)})-f(x_{l(k)}+s_{l(k)})}{m(x_{l(k)})-m(x_{l(k)}+s_{l(k)})},
	\end{equation*}
	which implies $\rho_{s_{l(k)}}\ge 1\ge\eta_{1}$ for all sufficiently large $l(k)\in\mathcal{S}$, which contradicts \eqref{re}. We complete the proof.
\end{proof}

\section{Numerical experiments}
In this section, the ideas introduced in Section \ref{TRBB} are first illustrated on two classical functions. The impact of the modified trust region radius update rule and the selection of regularization parameters on the robustness and efficiency of the Algorithm \ref{alg:TRBB} are then evaluated on 76 large-scale unconstrained optimization problems from CUTE \cite{Gould2015CUTEstConstrainedUnconstrained}. Finally, we apply the Algorithm \ref{alg:TRBB} to find spherical $t$-designs, which is a challenging numerical computation problem. All numerical experiments performed using \text{Matlab} R2019a under a Windows 11 operating system on a DELL G15 5200 personal computer with an Intel Core i7-12700H and 16 GB of RAM.

For all our tests, the parameters in Algorithm \ref{alg:TRBB} are as follows: $\Delta_{0}=1$, $\varrho=3$, $\alpha_{4}=0.25$, $\alpha_{1}=0.5$, $\alpha_{3}=1.5$, $\alpha_{2}=2$, $\eta_{1}=0.1$, $\eta_{2}=0.75$, $\eta_{3}=1.5$, $\eta_{4}=0.001$, $t_{min}=10^{-10}$, $t_{max}=10^{10}$. All together, the following four algorithms are tested and compared, \textbf{GBB}: BB method with non-monotone line search \cite{Raydan1997BarzilaiBorweinGradienta}, \textbf{RBBTR}: Algorithm \ref{alg:TRBB} with regularization parameter \eqref{parameters}, \textbf{RBBTRe}: Algorithm \ref{alg:TRBB} with regularization parameter \eqref{exp}, \textbf{BBTR}: Algorithm \ref{alg:TRBB} with $\alpha_{k}^{BB1}$ \eqref{BBlong}.
For consistency, all tested algorithms employ the same non-monotonic technique with $M=20$ in \eqref{Ck}, the same initial step $t_{1}=\frac{1}{\|g_{1}\|_{\infty}}$ and the same step interval $[t_{min}, t_{max}]$. The stopping condition for all the testing algorithms is $\|g_{k}\|\le10^{-6}(1+|f(x_{k})|)$. In addition, the algorithm is stopped if the number of iterations exceeds $20,000$. In such a case, we claim fail of this algorithm.  
\subsection{Test on two functions} 
In this part, we denote the \text{RBBTR} algorithm that does not include the \textit{too failed} case the  $\text{RBBTR}^{*}$. Similarly, we have  $\text{RBBTRe}^{*}$ and $\text{BBTR}^{*}$ algorithms. The ideas introduced in Section \ref{TRBB} are verified by comparing the performance of these algorithms on two classic functions, which are Extended White \& Holst function \eqref{WhiteHolst} and perturbed Tridiagonal Quadratic function \eqref{TridigonalQua} (see Andrei  \cite{Andrei2008UnconstrainedOptimizationTest}).
\begin{equation}\label{WhiteHolst}
f(x)=\sum_{i=1}^{n/2}\big(c(x_{2i}-x_{2i-1}^{3})^{2}+(1-x_{2i-1})^{2}\big),
\end{equation}
this function has a deep, curved valley following the cubic curve  $x_{2i}=x_{2i-1}^{3}$ for $i=1,\ldots,n/2$, and its minimizer is $x=(1,\ldots,1)^{\T}$. Here, we set $c=10^{4}$ and  $x_{0}=(-1.2,1\ldots,-1.2,1)^{\T}$.
\begin{equation}\label{TridigonalQua}
f(x)=x_{1}^2+\sum_{i=2}^{n-1}\big(ix_{i}^2+(x_{i-1}+x_{i}+x_{i+1})^2\big),
\end{equation}
the contour surface of this function is a hyper-ellipsoid, the condition number of this problem increases as $n$ increases, and its minimizer is $x=(0,\ldots,0)^{\T}$. The initial point $x_{0}=(0.5,\ldots,0.5)^{\T}$. We set $n=5000$ for the both problems. 
\begin{figure}[h!]
	\centering
	\subfigure[Extended White \& Holst function]{\includegraphics[width=0.5\linewidth]{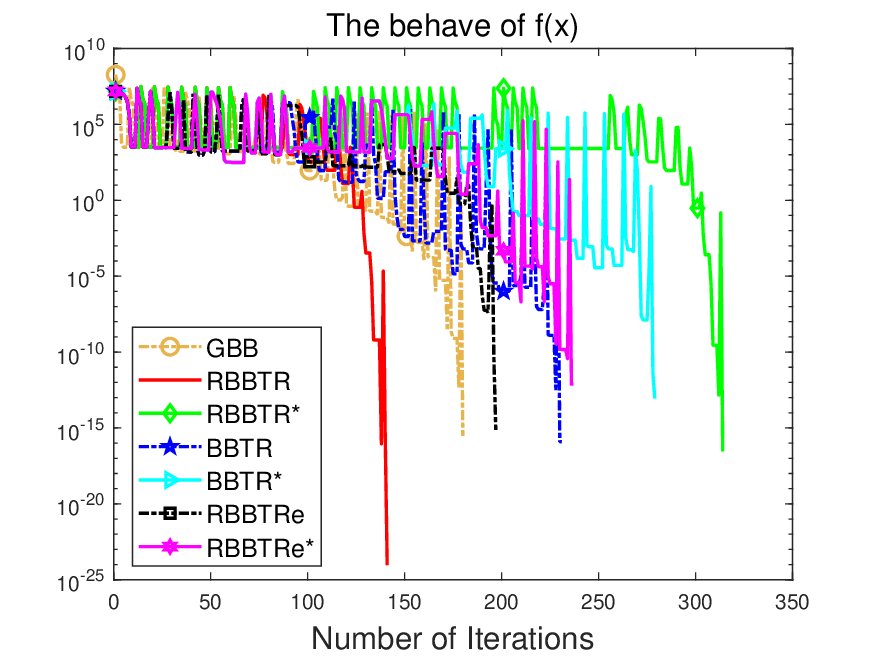}}\hspace{-8pt}
	\subfigure[Perturbed Tridiagonal Quadratic function]{\includegraphics[width=0.5\linewidth]{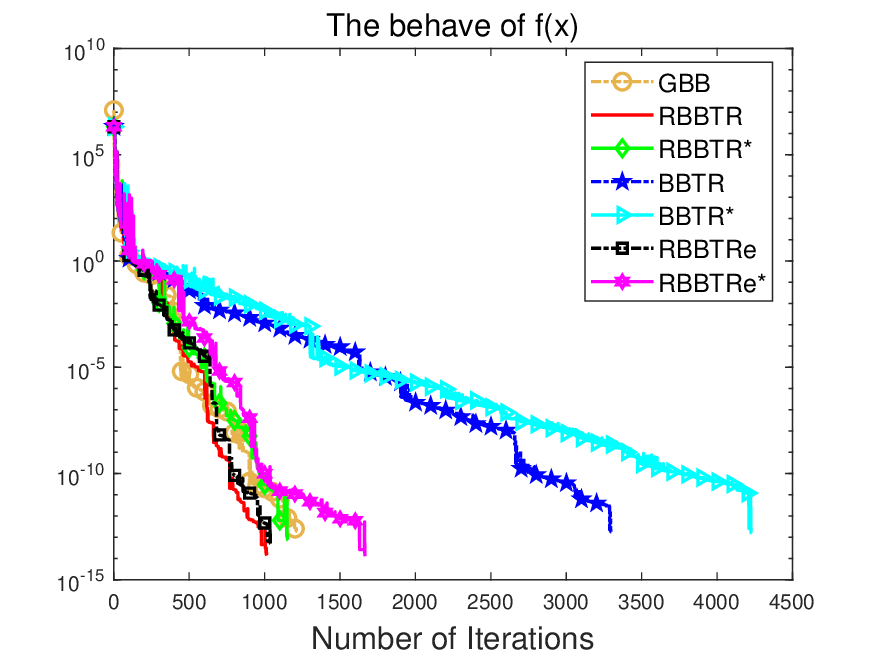}}
	\caption{\textit{Performance comparison for GBB, RBBTR, $\text{RBBTR}^{*}$, BBTR, $\text{BBTR}^{*}$, RBBTRe and $\text{RBBTRe}^{*}$ algorithms on functions \eqref{WhiteHolst} and \eqref{TridigonalQua}.}}
	\label{fig:WhiteAndPTQ}
\end{figure}

Figure \ref{fig:WhiteAndPTQ} shows that the number of iterations may be reduced by adding ``\textit{too failed}" case in trust-region framework, which improves the efficiency of algorithm. Meanwhile, under the framework of trust-region algorithm, the RBB step-size with regularization parameters related to the trust-region radius is more efficient than the of BB step-size. Meanwhile, in the Extended White \& Holst function, we can see that it is important to add the ``too failed" case for the \text{RBBTR} algorithm.  

\subsection{Performance on a test collection} 
We compare these different algorithms for the $76$ twice-continuously differentiable large-scale unconstrained optimization problems from \cite{Andrei2008UnconstrainedOptimizationTest}. Figure \ref{fig:Performance} presents performance profiles \cite{Dolan2002Benchmarkingoptimizationsoftware} obtained by RBBTR, RBBTRe and GBB, BBTR algorithms using CPU time and obtained gradient norm as metrics. For each algorithm, the vertical axis of the figure shows the percentage of problems the algorithm solves within the factor $\omega$ of the minimum value of the metric. It can be seen that RBBTRe performs better than RBBTR, BBTR and GBB.
\begin{figure}[h!]
	\centering
	\subfigure[CPU time]{\includegraphics[width=0.5\linewidth]{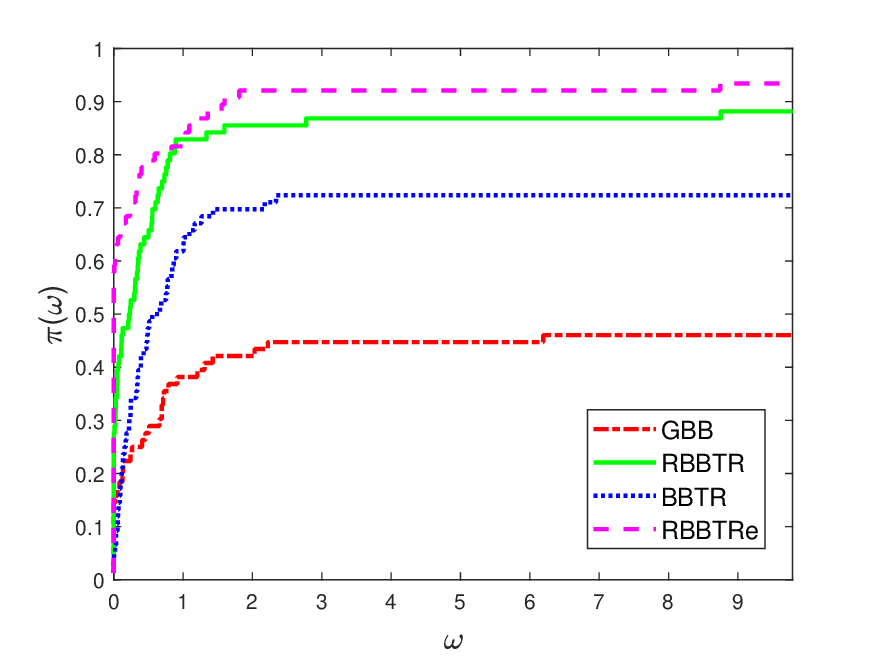}}\hspace{-8pt}
	\subfigure[Obtained gradient norm]{\includegraphics[width=0.5\linewidth]{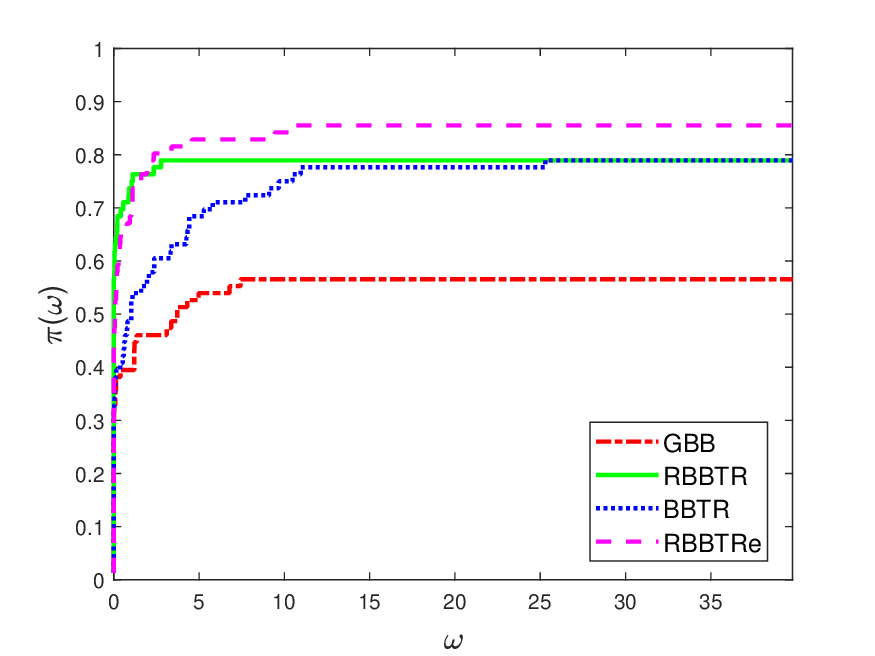}}
	\caption{\textit{Performance profiles of GBB, RBBTR, BBTR and RBBTRe algorithms on 76 unconstrained problems from \cite{Andrei2008UnconstrainedOptimizationTest}, CPU time (left) and obtained gradient norm (right) metrics.}}
	\label{fig:Performance}
\end{figure}

\subsection{Finding spherical $t$-designs efficiently}
In this part, we focus on the problem of finding a spherical $t$-design on $\mathbb{S}^{2}:=\{(x,y,z)^{\T}\in\mathbb{R}^{3}|x^2+y^2+z^2=1\}$, which was introduced by Delsarte et al.  \cite{Delsarte1991Sphericalcodesdesigns} in 1977.

A point set $\mathcal{X}=\{\mathtt{x}_{1},\ldots,\mathtt{x}_{N}\}\subset\mathbb{S}^{2}$ is a spherical $t$-design if it satisfies
\begin{equation*}
\frac{1}{N}\sum_{j=1}^{N}p(\mathtt{x}_{j})=\frac{1}{4\pi}\int_{\mathbb{S}^{2}}p(\mathtt{x})d\omega(\mathtt{x})\quad\quad \forall p\in\mathbb{P}_{t},
\end{equation*}
where $d\omega(\mathtt{x})$ denotes area measure on the unit sphere, $\mathbb{P}_{t}$ is the spherical polynomials with degree $t$. For $t\ge 1$, the existence of a spherical $t$-design was proved in \cite{Seymour1984Averagingsetsgeneralization}. Commonly, the interest is in the smallest number $N_{t}^{*}$ of points required to give a spherical $t$-design. A lot of important works on spherical $t$-design have been done \cite{Maier1999distributionpointssphere,Chen2006ExistenceSolutionsSystems,Chen2010Computationalexistenceproofs,Congpei2010WellConditionedSpherical}.  

In the influential paper  \cite{Sloan2009variationalcharacterisationspherical}, the authors proved that $\mathcal{X}$ is a spherical $t$-design if and only if 
\begin{equation}\label{variational}
A_{N,t}(\mathcal{X})=\frac{4\pi}{N^2}\sum_{n=1}^{t}\sum_{k=1}^{2n+1}(\sum_{i=1}^{N}Y_{n}^{k}(\mathtt{x}_{i}))^{2}=\frac{4\pi}{N^2}\sum_{i=1}^{N}\sum_{j=1}^{N}\sum_{n=1}^{t}\frac{2n+1}{4\pi}P_n(\mathtt{x}_{i}^{\T}\mathtt{x}_{j})=0,
\end{equation}
with $\{Y_{0}^{1},Y_{1}^{1},\ldots, Y_{t}^{2t+1}\}$ for degree $n=0,1,\ldots,t$ and order $k=1,\ldots,2n+1$ is a complete set of orthonormal real spherical harmonics basis for $\mathbb{P}_{t}$, $P_{n}$ the Legendre polynomial of degree $n$ normalized to $P_{n}(1)=1$. Therefore, the problem of finding a spherical $t$-design is expressed as solving a nonlinear and non-convex optimization problem
\begin{equation}\label{tDesign}
\min \  \{f(\mathcal{X}):=A_{N,t}(\mathcal{X})\}\quad\text{s.t.}\quad {\mathcal{X}\subset{\mathbb{S}^{2}}}.	
\end{equation}

Chen et al. \cite{Chen2010Computationalexistenceproofs} proved that spherical $t$-designs with $N=(t+1)^{2}$ points exist for all degrees $t$ up to 100 on $\mathbb{S}^{2}$. An et al. \cite{An2020Numericalconstructionspherical} numerically constructed spherical $t$-designs by solving \eqref{tDesign} using BB method with \text{Armijo-Goldstein} rule, we denote it \textbf{AGBB} method. And they numerically found the spherical $t$-designs $\mathcal{X}$ with $N=(t+1)^{2}$ up to $t=127$. 

Let the basis matrix be 
$$\mathbf{Y}_{t}^{0}(\mathcal{X})=
\begin{pmatrix}
\frac{1}{\sqrt{4\pi}} & \cdots & \frac{1}{\sqrt{4\pi}} \\
Y_{1}^{1}(\mathtt{x}_{1})& \cdots & Y_{1}^{1}(\mathtt{x}_{N})\\
\vdots & \ddots & \vdots \\
Y_{t}^{2t+1}(\mathtt{x}_{1}) & \cdots & Y_{t}^{2t+1}(\mathtt{x}_{N})
\end{pmatrix}
\in\mathbb{R}^{(t+1)^{2}\times N}.$$
From the result in \cite{An2020Numericalconstructionspherical}, we know that if  $\mathcal{X}\subset\mathbb{S}^{2}$ is a stationary point set of $A_{N,t}$ and the minimal singular value of basis matrix $\mathbf{Y}_{t}^{0}(\mathcal{X})$ is positive, then $\mathcal{X}$ is a spherical $t$-design. This result is used to check whether the obtained stationary point set is a spherical $t$-design.

Based on the code in \cite{An2020Numericalconstructionspherical}, we perform numerical experiments of RBBTR and RBBTRe algorithms. The termination condition of the algorithms is as follows
\begin{equation*}
\|g_{k}\|_{2}<\varepsilon_{1}\|g_{1}\|\ \text{or}\  |f_{k}-f_{k+1}|\le\varepsilon_{2} \ \text{or}\  \|\mathcal{X}_{k}-\mathcal{X}_{k+1}\|\le\varepsilon_{2}.
\end{equation*}
Here we set $\varepsilon_{1}=10^{-8}$, $\varepsilon_{2}=10^{-16}$. Initial step-size $t_{1}=1$. The number of maximum iterations and function evaluations are $10^{4}$ and one million. The parameters in Algorithm \ref{alg:TRBB} are consistent with those in the previous subsection. In the AGBB method, the parameters in \cite{An2020Numericalconstructionspherical} are used by default. The initial points in this experiment are consistent with those in \cite{An2020Numericalconstructionspherical}. We report in Table \ref{tab:spherical} the CPU time,  number of iterations, gradient norms $\|\nabla^{*}A_{N,t}(\mathcal{X})\|$, function values $A_{N,t}^{*}(\mathcal{X})$, and the minimal singular value $\sigma$ of basis matrix $\mathbf{Y}_{t}^{0}(\mathcal{X})$ for the AGBB, RBBTR and RBBTRe algorithms when the termination conditions are met for different degree $t$. 
\begin{table*}[h!]
	\centering
	\large 
	\renewcommand{\arraystretch}{1.2}
	\caption{Performance of AGBB, RBBTR and RBBTRe algorithms on finding  spherical $t$-designs.}
	\resizebox{\textwidth}{!}{
		\begin{tabular}{|cc|ccccc|ccccc|ccccc|}
			\hline
			\multicolumn{2}{|c}{Problem} & \multicolumn{5}{c}{AGBB}              & \multicolumn{5}{c}{RBBTR}            & \multicolumn{5}{c|}{RBBTRe} \\
			\hline
			t   & N     & Time  & \text{Iter}  & $\|\nabla A_{N,t}^{*}(\mathcal{X})\|$ & $A_{N,t}^{*}(\mathcal{X})$  & min($\sigma$) & Time  & \text{Iter}  & $\|\nabla A_{N,t}^{*}(\mathcal{X})\|$ & $A_{N,t}^{*}(\mathcal{X})$  & min($\sigma$) & Time  & \text{Iter}  & $\|\nabla A_{N,t}^{*}(\mathcal{X})\|$ & $A_{N,t}^{*}(\mathcal{X})$  & min($\sigma$) \\
			\hline
			10    & 121   & 0.154  & 100   & 8.87E-09 & 1.28E-15 & 1.3270  & 0.177  & 110   & 3.53E-09 & 2.01E-16 & 1.3260  & 0.215  & 138   & 3.15E-09 & 4.43E-16 & 1.3260  \\
			15    & 169   & 0.595  & 132   & 7.01E-09 & 2.03E-15 & 1.5323  & 0.844  & 189   & 8.00E-08 & 2.04E-16 & 1.5319  & 0.748  & 169   & 2.87E-08 & 4.77E-14 & 1.5319  \\
			20    & 441   & 2.852  & 253   & 7.58E-09 & 6.56E-16 & 1.7983  & 2.708  & 265   & 9.50E-09 & 2.27E-15 & 1.7990  & 2.165  & 210   & 1.05E-07 & 3.42E-14 & 1.7990  \\
			25    & 676   & 5.316  & 241   & 2.64E-08 & 1.00E-13 & 1.9018  & 5.575  & 285   & 3.58E-09 & 5.34E-15 & 1.9013  & 6.371  & 324   & 1.96E-08 & 6.76E-14 & 1.9013  \\
			30    & 961   & 11.480  & 274   & 1.11E-08 & 3.05E-14 & 1.9448  & 10.185  & 280   & 4.24E-08 & 6.64E-14 & 1.9456  & 12.233  & 340   & 1.17E-08 & 1.80E-14 & 1.9455  \\
			35    & 1296  & 47.154  & 437   & 6.01E-08 & 1.11E-12 & 2.0682  & 22.026  & 332   & 9.83E-09 & 1.48E-14 & 2.0691  & 33.575  & 390   & 1.58E-08 & 1.03E-14 & 2.0691  \\
			40    & 1681  & 185.847  & 741   & 1.62E-07 & 1.03E-11 & 2.0540  & 97.996  & 399   & 9.31E-09 & 1.83E-15 & 2.0536  & 114.021  & 462   & 2.04E-09 & 1.06E-14 & 2.0536  \\
			45    & 2116  & 388.226  & 730   & 6.28E-07 & 1.75E-10 & 2.1196  & 145.692  & 407   & 6.79E-08 & 2.00E-13 & 2.1197  & 134.339  & 344   & 2.53E-08 & 2.65E-13 & 2.1197  \\
			50    & 2601  & 249.369  & 375   & 1.61E-08 & 1.14E-13 & 2.3393  & 248.470  & 405   & 5.75E-08 & 3.03E-13 & 2.3385  & 293.445  & 483   & 2.19E-08 & 2.30E-13 & 2.3385  \\
			55    & 3136  & 480.553  & 444   & 3.09E-08 & 6.22E-13 & 2.4334  & 403.302  & 404   & 3.88E-08 & 5.71E-13 & 2.4345  & 504.808  & 533   & 9.32E-09 & 1.46E-14 & 2.4345  \\
			60    & 3721  & 639.888  & 390   & 8.49E-08 & 6.26E-12 & 2.3639  & 541.773  & 370   & 7.17E-08 & 4.30E-12 & 2.3642  & 649.936  & 453   & 4.29E-08 & 1.09E-12 & 2.3640  \\
			65    & 4356  & 1634.305  & 673   & 1.48E-08 & 6.14E-14 & 2.1894  & 1346.564  & 579   & 1.47E-08 & 1.29E-13 & 2.1902  & 1385.253  & 570   & 1.36E-08 & 1.29E-13 & 2.1902  \\
			70    & 5041  & 2021.140  & 598   & 2.78E-08 & 3.76E-13 & 2.1090  & 1534.415  & 478   & 6.21E-08 & 3.54E-12 & 2.1084  & 1705.671  & 489   & 2.78E-08 & 6.05E-13 & 2.1084  \\
			75    & 5776  & 2517.164  & 568   & 9.79E-09 & 5.28E-15 & 2.2955  & 2049.604  & 491   & 9.31E-08 & 4.27E-12 & 2.2973  & 2074.003  & 517   & 6.04E-08 & 2.80E-12 & 2.2973  \\
			80    & 6561  & 6138.994  & 802   & 1.73E-07 & 3.99E-11 & 2.3504  & 2706.260  & 500   & 9.60E-08 & 1.20E-11 & 2.3499  & 2686.841  & 496   & 9.97E-08 & 1.39E-11 & 2.3499  \\
			84    & 7225  & 4804.699  & 605   & 1.20E-07 & 8.41E-13 & 1.7980  & 3784.836  & 556   & 2.75E-08 & 1.59E-13 & 1.7989  & 5609.035  & 702   & 3.56E-08 & 1.46E-12 & 1.7989  \\
			88    & 7921  & 8294.013  & 870   & 6.70E-07 & 2.24E-11 & 2.1242  & 5024.556  & 590   & 3.64E-08 & 1.54E-12 & 2.1229  & 5104.978  & 622   & 7.78E-08 & 8.35E-12 & 2.1229  \\
			96    & 9409  & 13817.921  & 863   & 1.87E-08 & 5.99E-13 & 2.1647  & 9933.572  & 734   & 3.16E-08 & 1.52E-13 & 2.1650  & 9145.463  & 673   & 4.79E-08 & 4.84E-12 & 2.1650  \\
			127   & 16384 & 52940.230  & 802   & 6.21E-07 & 1.28E-09 & 1.9674  & 61771.828  & 991   & 1.36E-07 & 3.03E-14 & 1.9668  & 49783.287  & 790   & 1.41E-07 & 7.70E-11 & 1.9669  \\
			\hline
	\end{tabular}}
	\label{tab:spherical}
\end{table*}

From the results in Table \ref{tab:spherical}, we can see that for small $t$, \text{AGBB} has an advantage in CPU time, but as the degree $t$ gradually increases, \text{RBBTR} and \text{RBBTRe} show significant advantages. Meanwhile, the minimal singular value of the basis matrix corresponding to the stationary point set obtained by the algorithms is positive for different $t$, which verifies the effectiveness of the three algorithms.

Figure \ref{fig:PerformanceSph} presents the performance profiles of the three algorithms using CPU time, iteration numbers and obtained gradient norm as metrics.   
\begin{figure}[h!]
	\centering
	\subfigure[CPU time]{\includegraphics[width=0.33\linewidth]{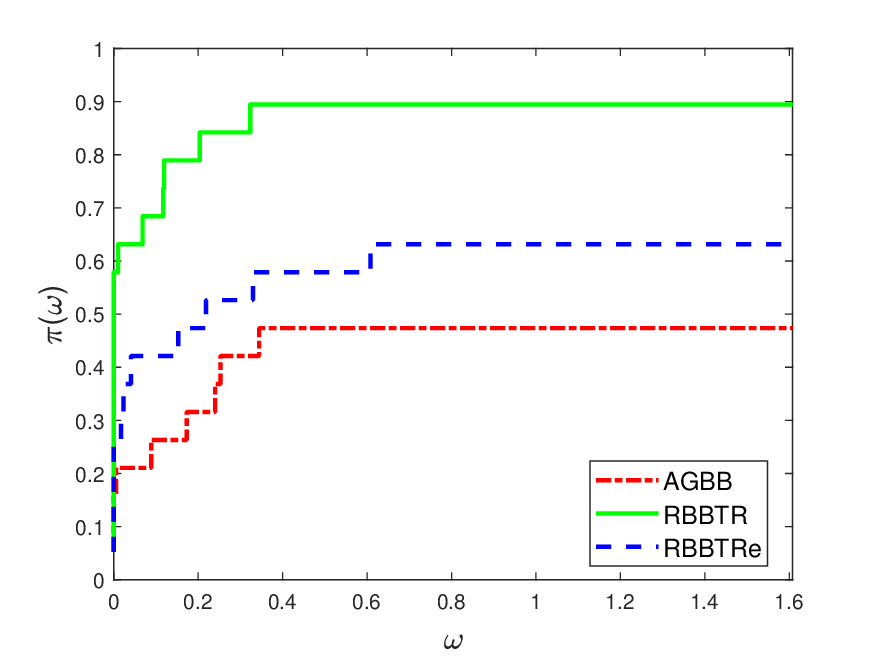}}\hspace{-6pt}
	\subfigure[Iteration number]{\includegraphics[width=0.33\linewidth]{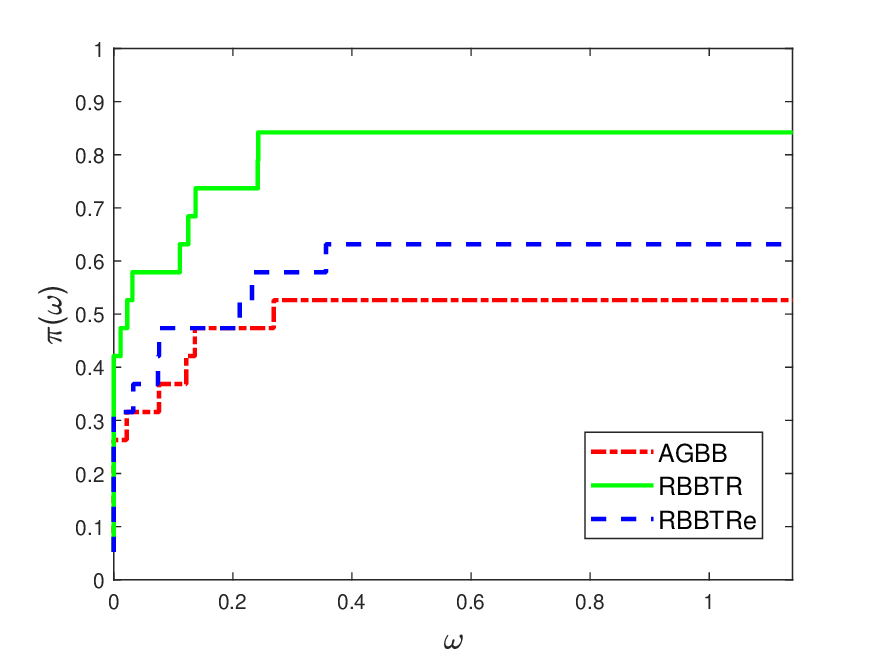}}\hspace{-6pt}
	\subfigure[Optimal gradient norm]{\includegraphics[width=0.33\linewidth]{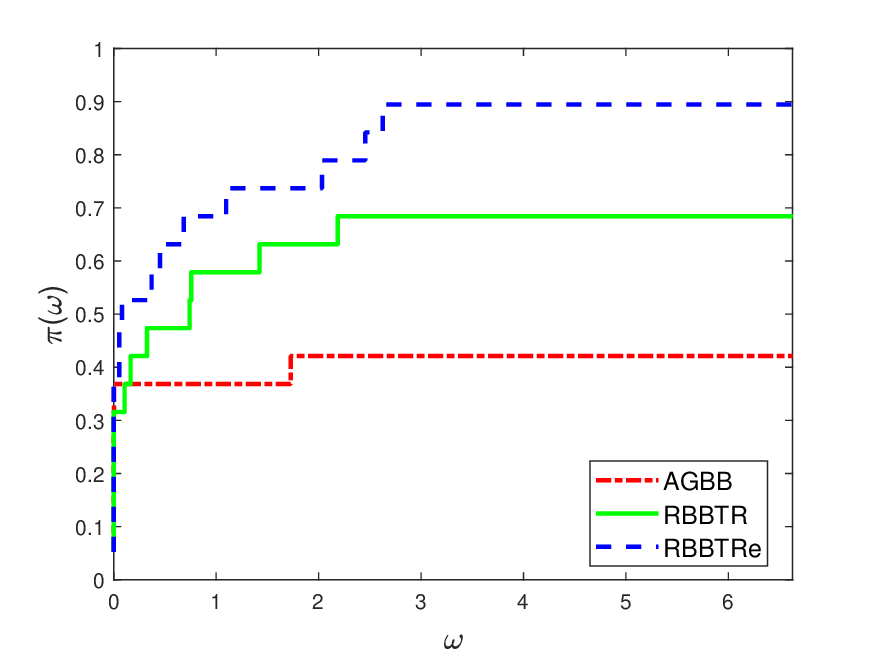}}
	\caption{\textit{Performance profiles of AGBB, RBBTR, RBBTRe algorithms on spherical $t$-designs.}}
	\label{fig:PerformanceSph}
\end{figure}

It can be seen that, in terms of CPU time and Iteration number, \text{RBBTR} performs better than \text{RBBTRe} and better than \text{AGBB}, in terms of optimal gradient norm $\|\nabla A_{N,t}^{*}(\mathcal{X})\|$ obtained, \text{RBBTRe} performs better than \text{RBBTR} and better than \text{AGBB}.

\section{Concluding remarks}
In this paper, we develop a Trust Region method with Regularized \BB step-size obtained in a previous paper for solving large-scale unconstrained optimization problems. Simultaneously, the non-monotone technique is combined to formulate an efficient trust region method. The proposed method adaptively generates a suitable step-size within the trust region. The minimizer of the resulted model can be easily determined, and at the same time, the convergence of the algorithm is also maintained. We refine the update rule of the trust region radius and use the relationship between the trust region and the regularization parameter to obtain two types of regularization parameters. Some numerical experiments verify the performances of the RBBTR and RBBTRe algorithms with adaptive regularization parameter. How to more accurately analyze the relationship between the regularization parameter and the performance of the trust region algorithm requires further exploration.

%
%

\section{Declarations}

\subsection{Ethical Approval}
Not Applicable.
\subsection{Availability of supporting data}
Data available on request from the authors.
\subsection{Competing interests}
The authors declare that there is no conflict of interest.
\subsection{Funding}
The work was supported by the National Natural Science Foundation of China (Project No. 12371099).


\bibliography{TRBB}

\end{document}